\documentclass{elsart}
\usepackage{amsfonts}
\usepackage{amsmath}

\input{tcilatex}

\begin{document}

\begin{frontmatter}%

{\LARGE Three remarks on Matula numbers\footnote{%
Preliminary draft, please do not quote without the author permission. I
thank Mar\'{\i}a F. Morales for helpful comments. I also thank Regis de la
Bret\`{e}che who have pointed out to me an important reference, from which I
borrow the term of GIM function. Of course, all errors and/or omissions are
my own. Financial support from the Spanish Ministry of Education through the
project ECO2010-21624 is gratefully acknowledged. }\bigskip }

Albert Burgos

Universidad de Murcia

30100 Espinardo, Murcia, Spain

e-mail: \texttt{albert@um.es}

\begin{abstract}%
In \textit{SIAM Review} \textbf{10}, page 273, D. W. Matula described a
bijection between $\mathbb{N}$ and the set of topological rooted trees; the
number is called the Matula number of the rooted tree. The Gutman-Ivi\'{c}%
-Matula (GIM)\ function $g(n)$ computes the number of edges of the unique
tree with Matula number $n$. Since there is a prefix-free code for the set
of prime numbers such that the codelength of each prime $p$ is $2g(p)$, we
show how some properties of the GIM function can be obtained trivially from
coding theorems. 
\end{abstract}%


\end{frontmatter}%

\section*{Introduction}

$\mathbb{N}$ stands for the set of natural numbers, $\mathbb{W}$ for the set
of whole numbers $\mathbb{N\cup }\{0\}$, $\mathcal{T}$ for the set of finite
and undirected rooted trees, and $\mathbb{T}$ for the set of topological
rooted trees (i.e. all equivalence classes of undirected rooted trees where
the equivalence is the natural isomorphism). Let $S\subset \mathbb{N}$
(respectively $\mathbb{W})$ be a set defined by property such that every
natural number (respectively whole number) has a unique decomposition as a
multisubset of $S$. Since the set of integers is totally ordered, then $S$
is totally ordered. If the least element in $S$ is greater than 1 (resp. $0$%
), the index of an element $s_{i}\in S$ is less than $s$, i.e. $s_{i}\in S$
implies $i<s_{i}$. This leads to a recursive map from the naturals into $%
\mathbb{T}$. We can easily construct examples of such property by setting $%
S=\{b^{n}|n\geq 0\}$ (for some base $b>1$) or $S=\{n!|n\geq 1\}$. In these
two cases, decomposition builds on the addition operation on $\mathbb{W}$.
However, the prototype of all $S$ defined as above is of course the set of
all primes $\mathbb{P=}\left\{ p(n)|n\geq 1\right\} $, for which
decomposition is the product of elements in $\mathbb{P}$. In this case, the
fundamental theorem of arithmetic tells us that the recursive map is a
denumeration of $\mathbb{T}$, as was noted independently by Matula \cite%
{Mat68} and G\"{o}bel \cite{Gob80}, which give an explicit construction of
the map and its inverse. Thus, any statistics on rooted trees may be
assigned to a natural number. This has been done for several statistics
starting with the pioneering work of Gutman, Ivi\'{c} and Elk \cite%
{GutIvicElk83}, with follow-ups by Gutman and Yeh \cite{GutY93}, Gutman and
Ivi\'{c} \cite{GutIvic94}, \cite{GutIvic96}, and Deutsch \cite{Deu12}, among
others. In particular, Gutman, Ivi\'{c} and Elk ask themselves how many
edges does a natural number have, proving that the solution is given by the
unique function completely additive $g$, such that, for any $n\in \mathbb{N}$%
,%
\begin{equation}
g\left( p(n)\right) =1+g\left( n\right) \text{.}
\end{equation}%
This function, dubbed by La Bret\`{e}che and Tenenbaum in \cite{DLBT00} as
the \textit{Gutman-Ivi\'{c}-Matula} (\textit{GIM) function}, has a simple
translation in terms of coding theory: It represents the semi-length of the
codewords when natural numbers are encoded a Dyck alphabet via the Matula
bijection. By exploiting this relation we can apply basic results on
information entropy to provide non-asymptotic results on its behavior.

\section{Preliminaries}

Write $\pi (n)$ for the number of primes less than or equal to $n$, and $p(n)
$ for the $n$-th prime in ascending order. Given $n\in 
\mathbb{N}
_{\geq 2}$, if the prime factorization of $n$ is $f_{1}\cdot f_{2}\cdot
\cdots \cdot f_{m}$ and $\left\{ {\Huge \circ }\right\} $ denotes the rooted
tree consisting of a single node, then the function $\tau :\mathbb{N}\mapsto 
\mathbb{T}$ is defined as a recursion:\smallskip 

\begin{enumerate}
\item $\tau (1)=\left\{ {\Huge \circ }\right\} $.

\item For $n\geq 2$, $\tau (n)$ is the tree in which the root is adjacent to
the roots of $\tau (\pi (f_{i}))$ for $1\leq $ $i\leq m$.
\end{enumerate}

The map $\tau $ is a bijection. Its inverse is defined as follows:\smallskip 

\begin{enumerate}
\item $\tau ^{-1}({\Huge \circ })=1$.

\item If the root of tree $\mathbf{t}$ is adjacent to subtrees $\mathbf{t}%
_{1},\mathbf{t}_{2},\ldots ,\mathbf{t}_{m}$, then 
\begin{equation*}
\tau ^{-1}(\mathbf{t})=\prod_{i=1}^{m}{p(\tau ^{-1}(\mathbf{t}_{i}))}. 
\end{equation*}
\end{enumerate}

A prime factorization of a number is in \emph{canonical order} when the
primes are presented in nondecreasing order. An analogue for rooted trees
goes as follows: If $\tau (n)=\mathbf{t}$, the rooted tree $\mathbf{t}$ is
presented \emph{canonically} when:

\begin{enumerate}
\item The rooted trees $\mathbf{t}_{1},\mathbf{t}_{2},\ldots ,\mathbf{t}_{m}$%
, corresponding to the factors $f_{1},f_{2},\cdots ,f_{m}$, respectively,
are presented from left to right.

\item Each rooted tree $\mathbf{t}_{i}$ is presented canonically.
\end{enumerate}

Notice that $n\in \mathbb{P}$ iff $\tau (n)\in \mathbb{T}_{p}$, where $%
\mathbb{T}_{p}\subset $ $\mathbb{T}$ is the set of all planted rooted trees.%
\footnote{%
i.e. rooted trees for which the root has degree 1.}

Let $(\mathbb{N},\cdot ,1)$ denote the commutative monoid of the positive
integers under product. Let the \emph{merging} of rooted trees $\mathbf{t}%
_{1}$ and $\mathbf{t}_{2}$, denoted, $\mathbf{t}_{1}\wedge \mathbf{t}_{2}$,
to be the labeled tree that results from identifying their roots. Now let $(%
\mathbb{T},\wedge ,{\Huge \circ })$ be the commutative monoid of rooted
trees under product, and let $\mathbf{p}\left( \mathbf{t}\right) $ be the
rooted tree whose root is adjacent to subtree $\mathbf{t}$. Since $\tau $ is
a bijection, there is a isomorphism between $(\mathbb{N},\cdot ,1)$ and $(%
\mathbb{T},\wedge ,{\Huge \circ })$, which trivially extends to an
isomorphism of the algebra $\mathcal{N}=\left\langle \mathbb{N},\cdot
,p,1\right\rangle $ onto the algebra $\mathcal{T}=\left\langle \mathbb{T}%
,\wedge ,\mathbf{p},{\Huge \circ }\right\rangle $.

\section{The Matula code}

In what follows, we are motivated to use the algebra $\mathcal{T}$ to encode
the positive integers. Any tree can be written as a string of symbols from
the set $\left\{ \mathbf{p},{\Huge \circ },\wedge ,(,)\right\} $. In any
string representing a tree presented in canonical form, $\mathbf{p}$ is
always followed by $($, the operator $\wedge $ is always followed by $%
\mathbf{p}$, and ${\Huge \circ }$ appears only between $($ and $)$. Hence,
no information is lost if we drop all the $\mathbf{p}$'s, ${\Huge \circ }$%
's, and $\wedge $'s. For example, the rooted tree (b) in the figure above is
represented as $\mathbf{p}(\mathbf{p}(\mathbf{p}({\Huge \circ }{\small )}%
\wedge \mathbf{p}({\Huge \circ }{\small )))}$, which can be simply written
as $\left( \left( \left( {}\right) \left( {}\right) \right) \right) $.

We need now to introduce some relevant formalism. I shall call \emph{alphabet%
} a given set $\Sigma $ such that $2\leq \left\vert \Sigma \right\vert
<+\infty $. Elements of $\Sigma $ are called \emph{symbols}. A \emph{string }%
or\emph{\ word} over $\Sigma $ is any finite sequence of symbols from $%
\Sigma $. The length of a string is the number of symbols in the string (the
length of the sequence) and can be any non-negative integer. The empty
string is the unique string over $\Sigma $ of length $0$, and is denoted $%
\mathbf{e}$. The set of all strings over $\Sigma $ of length $t$ is denoted $%
\Sigma ^{t}$. (Note that $\Sigma ^{0}=\left\{ \mathbf{e}\right\} $ for any $%
\Sigma $.) The set of all finite-length strings over $\Sigma $ is the Kleene
closure of $\Sigma $, denoted $\Sigma ^{\ast }$. For any two strings $%
\mathbf{s}$ and $\mathbf{s}^{\prime }$ in $\Sigma ^{\ast }$, their
concatenation is defined as the sequence of symbols in $\mathbf{s}$ followed
by the sequence of symbols in $\mathbf{s}^{\prime }$, and is denoted $%
\mathbf{ss}^{\prime }$. The empty string serves as the identity element; for
any string $\mathbf{s}$, $\mathbf{es}=\mathbf{se}$ $=\mathbf{s}$. Therefore,
the set $\Sigma ^{\ast }$ and the concatenation operation form a monoid, the
free monoid generated by $\Sigma $. In addition, the length function defines
a monoid homomorphism $\ell :\Sigma ^{\ast }\rightarrow \mathbb{W}$. Given a
set $\mathcal{S}$ with Kleene closure $\mathcal{S}^{\ast }$, a \emph{code}
is a function $c:\mathcal{S}\rightarrow \Sigma ^{\ast }$. The elements of $c(%
\mathcal{S})$\ will be referred to as the \emph{codewords}, and $c$\ is said
to be a $\left\vert \Sigma \right\vert $-ary code. In this paper we shall
assume that any code $c$ satisfies that if $\mathbf{s},\mathbf{s}^{\prime },%
\mathbf{ss}^{\prime }\in c(\mathcal{S})$ and $\mathbf{s}\neq \mathbf{e}$,
then $\mathbf{s}^{\prime }=\mathbf{e}$ (i.e. the code is \emph{prefix-free}).

In accordance with the above definitions, by setting $\Sigma =\left\{
(,)\right\} $ and $\mathcal{S}=\mathbb{P}$, it is clear that $\tau $ induces
a code, which we dub the Matula code, denoted $c_{M}$.

\section{Concluding remarks}

The following are results are the direct translation of the coding approach
into properties of the Gutman-Ivi\'{c}-Matula arithmetic function.

In Gutman, Ivi\'{c} and Elk \cite{GutIvicElk83}, Gutman and Yeh \cite{GutY93}%
, Gutman and Ivi\'{c} \cite{GutIvic94}, and Gutman and Ivi\'{c} \cite%
{GutIvic96} is introduced an arithmetic function capturing some graph
theoretic properties of $\tau $:

\begin{definition}
Let $g:{\Bbb N}\mapsto {\Bbb C}$\ be the unique function completely additive
such that, por any $n\in {\Bbb N}$,%
\begin{equation}
g\left( p(n)\right) =1+g\left( n\right) \text{.}
\end{equation}
\end{definition}

Then, we have:

\begin{lem}
For all $p\in \mathbb{P}$, $\ell \circ c_{M}\left( p\right) =2g(p).$
\end{lem}

\begin{proof}
It trivially follows from Theorem 3(a) in Gutman and Ivi\'{c} \cite%
{GutIvic94}.
\end{proof}

Then, we have,

\begin{description}
\item[Conclusion 1:] $M_{\mathbb{P}}=\sum\nolimits_{p\in \mathbb{P}%
}1/4^{g(p)}<1/2.$
\end{description}

\begin{proof}
By Claim 1 and Kraft's inequality.
\end{proof}

And,

\begin{description}
\item[Conclusion 2:] $M_{\mathbb{N}}=\sum\nolimits_{n\in \mathbb{N}%
}1/4^{g(n)}=M_{\mathbb{P}}/4<2.$
\end{description}

\begin{proof}
By Conclusion 1 and equation (2).
\end{proof}

Thus, $M_{\mathbb{P}}$ is a probability, whereas $M_{\mathbb{P}}$ is not,
because $g(2)=1$ and therefore $M_{\mathbb{P}}>1/4$.

Bounds on $g$ appear in Gutman and Ivi\'{c} \cite{GutIvic94}, and \cite%
{GutIvic96}, which show, for all $n\geq 7$, 
\begin{equation}
\text{\textit{\b{g}}}(n)=\frac{\ln n}{\ln (\ln n)}\leq g(n)\leq \frac{3\ln n%
}{\ln 5}=\bar{g}(n).
\end{equation}%
The bounds \textit{\b{g} }and\textit{\ }$\bar{g}$ cannot be improved
upon---i.e. there exist infinite values of $n$ for which they are reached
asymptotically. However, Conclusion 2 implies that most values of $g$ are
close to those of $\bar{g}$. Indeed, in \cite{DLBT00}, La Bret\`{e}che and
Tenenbaum develop a method which evaluates the moments of $g$, showing that
if $G\left( n\right) =\tsum\limits_{1\leq i\leq n}\,g(i)$, then%
\begin{equation*}
G\left( n\right) =\phi n\ln n+O\left( n\ln (\ln n\right) )
\end{equation*}%
for some constant $\phi >0$ (see \cite{DLBT00} Theorem 1 and Theorem 3).
Now, we shall see that Shannon's source coding theorem \cite{Sh48} offers a
proof of the fact that for all $n$, $G\left( n\right) >\phi n\ln n$ whenever 
$\phi \leq 1/\ln 4$.

Let $S$ be a finite subset of $\mathbb{N}$, $2^{S}$ the associated $\sigma $%
-algebra, and $\mu =\left( \mu \left( s\right) \right) _{s\in S}$ a
probability distribution on the elements of $S$. We shall denote by $\mathbf{%
S}$ the discrete random variable taking values\ in $S$ according to $\mu $.
Thus, the expectation for the length of $c_{M}(\mathbf{S})$ is 
\begin{equation}
E\left[ \ell \circ c_{M}(\mathbf{S})\right] =2\tsum\limits_{i=1}^{n}\mu
\left( i\right) \,g(i)\text{.}
\end{equation}

Shannon's coding theorem places a lower bound on this expected length.
Namely, the ratio between the \emph{entropy} of $\mathbf{S}$, $H\left( 
\mathbf{S}\right) =-\tsum\nolimits_{s\in S}\mu \left( s\right) \log \mu
\left( s\right) ,$\footnote{%
Here and thereafter we denote by $\log $ and $\ln $ the logarithms in bases $%
2$ and $e$, respectively.} and the cardinal of the target alphabet $\Sigma $%
. Therefore 
\begin{equation}
E\left[ \ell \circ c_{MG}(\mathbf{S})\right] \geq H\left( \mathbf{S}\right)
=-\tsum\limits_{s\in S}\mu \left( s\right) \log \mu \left( s\right) ,
\label{ShB1}
\end{equation}%
i.e. 
\begin{equation}
2\tsum\limits_{s\in S}\mu \left( s\right) \,g(s)\geq \tsum\limits_{s\in
S}\mu \left( s\right) \log \mu \left( s\right)   \label{ShB2}
\end{equation}%
Trivially, $H\left( \mathbf{S}\right) \leq \log \left\vert S\right\vert $
with equality if and only if $\mu $ is the uniform distribution. Thus, we
get:

\begin{description}
\item[Conclusion 3:] \textit{For all} $n\in \mathbb{N}$,%
\begin{equation}
G\left( n\right) \geq \frac{n}{\ln 4}\ln n\text{.}
\end{equation}
\end{description}

\begin{proof}
Set $n\in {\Bbb N}$. If $S=\left[ n\right] $ and $\mu (i)=1/n$ for all $i\in
S$, (3.6) yieldst $G\left( n\right) \geq n\log n/2=n\ln n/\ln 4$.
\end{proof}

\end{document}